\documentclass[12pt]{amsart}
\usepackage{amsmath,amsthm,amssymb,fullpage,enumerate}
\usepackage{epsfig}
\usepackage[normalem]{ulem}
\usepackage[usenames,dvipsnames]{color}

\newcommand{\size}[1]{\left \vert #1 \right \vert}
\newcommand{\ceil}[1]{\left \lceil #1 \right \rceil}
\newcommand{\floor}[1]{\left \lfloor #1 \right \rfloor}
\newcommand{\prob}[1]{\mathrm{Pr}\left [ #1 \right ]}
\newcommand{\expect}[1]{\mathbb{E} \left [ #1 \right ]}
\newcommand{\eps}{\varepsilon}

\newcommand{\N}{{\mathbb N}}
\newcommand{\G}{{\mathcal G}}
\newcommand{\bg}{{b_g}}

\newtheorem{theorem}{Theorem}
\newtheorem{lemma}[theorem]{Lemma}

\newtheorem{proposition}[theorem]{Proposition}
\newtheorem{corollary}[theorem]{Corollary}
\newtheorem{observation}[theorem]{Observation}

\theoremstyle{definition}
\newtheorem{definition}[theorem]{Definition}

\begin{document}

\title{Game brush number}

\author{William B. Kinnersley}
\address{Department of Mathematics, University of Rhode Island, Kingston, RI, USA, 02881}
\email{\tt billk@mail.uri.edu}

\author{Pawe{\l} Pra{\l}at}
\address{Department of Mathematics, Ryerson University, Toronto, ON, Canada, M5B 2K3}
\email{\tt pralat@ryerson.ca}

\begin{abstract}
We study a two-person game based on the well-studied brushing process on graphs.  Players Min and Max alternately place brushes on the vertices of a graph.  When a vertex accumulates at least as many brushes as its degree, it sends one brush to each neighbor and is removed from the graph; this may in turn induce the removal of other vertices.  The game ends once all vertices have been removed.  Min seeks to minimize the number of brushes played during the game, while Max seeks to maximize it.  When both players play optimally, the length of the game is the {\em game brush number} of the graph $G$, denoted $\bg(G)$.

By considering strategies for both players and modelling the evolution of the game with differential equations, we provide an asymptotic value for the game brush number of the complete graph; namely, we show that $\bg(K_n) = (1+o(1)) n^2/e.$ Using a fractional version of the game, we couple the game brush numbers of complete graphs and the binomial random graph $\G(n,p)$. It is shown that for $pn \gg \ln n$ asymptotically almost surely $\bg(\G(n,p))=(1+o(1))p\bg(K_n)=(1+o(1))pn^2/e$. Finally, we study the relationship between the game brush number and the (original) brush number.
\end{abstract}

\maketitle

\section{Introduction}

Imagine a network of pipes that must be periodically cleaned of a regenerating contaminant, say algae. In {\em cleaning} such a network, there is an initial configuration of brushes on vertices, and every vertex and edge is initially regarded as \textit{dirty}.  A vertex is ready to be cleaned if it has at least as many brushes as incident dirty edges.  When a vertex is cleaned, it sends one brush along each incident dirty edge; these edges are now said to be \textit{clean}.  (No brush ever traverses a clean edge.)  The vertex is also deemed clean. Excess brushes remain on the clean vertex and take no further part in the process. (In fact, for our purposes in this paper, we may think about clean vertices as if they were removed from the graph.) The goal is to clean all vertices (and hence also all edges) of the graph using as few brushes as possible.  The minimum number of brushes needed to clean a graph $G$ is the {\em brush number} of $G$, denoted $b(G)$.

Figure~\ref{fig:ex1} illustrates the cleaning process for a graph $G$ where there are initially $2$ brushes at vertex $a$. The solid edges indicate dirty edges while the dotted edges indicate clean edges. For example, the process starts with vertex $a$ being cleaned, sending a brush to each of vertices $b$ and $c$.  

\begin{figure}[htbp]
\[\includegraphics[width=0.9\textwidth]{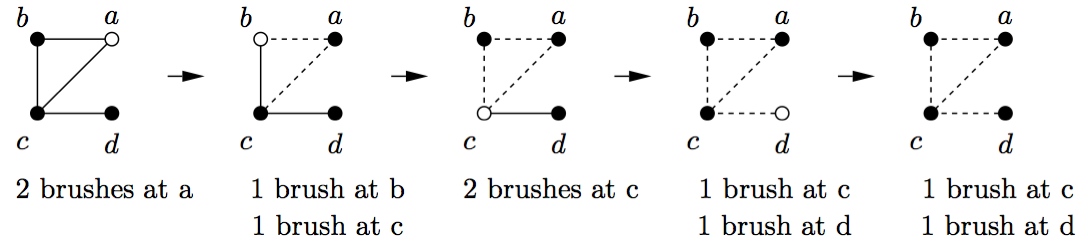}\] \label{fig:ex1} 
\vspace{-0.2in}\caption{An example of the cleaning process for graph $G$.}
\end{figure}

This model, which was introduced in~\cite{MNP}, is tightly connected to the concept of {\em minimum total imbalance} of a graph, which is used in graph drawing theory.  The cleaning process has been well studied, especially on random graphs~\cite{APW, P}. (See also~\cite{algorithm} for algorithmic aspects, \cite{MNP2, P2} for a related model of cleaning with brooms, \cite{BFGPP} for a variant with no edge capacity restrictions, \cite{BDP} for a variant in which vertices can send out no more than $k$ brushes, and~\cite{GNP} for a combinatorial game.) Owing to inspiration from chip-firing processes~\cite{chip, merino}, brushes disperse from an individual vertex in unison, provided that their vertex meets the criteria to be cleaned.  Models in which multiple vertices may be cleaned simultaneously are called {\em parallel cleaning models}; see~\cite{GMNP} for more details.  In contrast, {\em sequential cleaning models} mandate that vertices get cleaned one at a time. The variant considered in~\cite{MNP} and the one we consider in this paper are sequential in nature.

The {\em brushing game} that we introduce in this paper is a two-player game played on a graph $G$.  Initially, every vertex and edge is dirty and there are no brushes on any vertices.  The players, Max and Min, alternate turns; on each turn, a player adds one brush to a vertex of his or her choosing.  When a vertex accumulates at least as many brushes as it has dirty neighbors, it {\em fires}, sending one brush to each dirty neighbor. All edges incident to this vertex become clean, and the vertex itself becomes clean. This may in turn make other vertices ready to fire, so the process continues until we obtain a stable configuration. (It is known that the sequence in which vertices fire does not affect the distribution of brushes on dirty vertices---see below for more details.) The game ends when all vertices (and so all edges as well) are clean.  Max aims to maximize the number of brushes played before this point, while Min aims to minimize it.  When Min starts and both players play optimally, the length of the game on $G$ is the {\em game brush number} of $G$, denoted $\bg(G)$; we use $\widehat \bg(G)$ to denote the variant of the game in which Max plays first. 

The game brush number follows in the same spirit as the game matching number~\cite{CKOW}, game chromatic number~\cite{DZ}, game domination number~\cite{BKR}, toppling number~\cite{BKP}, etc., in which players with conflicting objectives together make choices that produce a feasible solution to some optimization problem.  The general area can be called \emph{competitive optimization}.  Competitive optimization processes can also be viewed as on-line problems in which one wants to build a solution to some optimization problem, despite not having complete control over its construction.  In this context, the game models adversarial analysis of an algorithm for solving the problem: one player represents the algorithm itself, the other player represents the hypothetical adversary, and the outcome of the game represents the algorithm's worst-case performance.

Throughout this paper, we consider only finite, simple, undirected graphs. For background on graph theory, the reader is directed to~\cite{west}.

\subsection{Main results}

Let us start with the following convenient bound proved in Section~\ref{sec:pre}.

\begin{theorem}\label{thm:at_most_one}
Always $\size{\bg(G) - \widehat \bg(G)} \le 1$.
\end{theorem}

Theorem~\ref{thm:at_most_one} is best possible.  For the star $K_{1,3k-1}$, we have $\bg(K_{1,3k-1}) = 2k-1$ and $\widehat \bg(K_{1,3k-1}) = 2k$; in particular, we have $\bg(P_3) = 1$ and $\widehat \bg(P_3) = 2$.  On the other hand, for $n \ge 3$ we have $\bg(C_n)=3$ and $\widehat \bg(C_n)=2$.  (We remark without proof that the graph $G_k$ obtained by taking $5k$ triangles and identifying a single vertex of each has $\bg(G_k) = 8k+1$ and $\widehat \bg(G_k) = 8k$; this yields another family---with unbounded brush number---witnessing sharpness of the bound.  We leave the details to the reader.)

For the remainder of the paper, we will be concerned primarily with the asymptotics of $\bg$ over various families of graphs.  Hence the difference between $\bg$ and $\widehat \bg$ is unimportant; we use whichever is most convenient (but prefer $\bg$ in general).

Because the game produces a feasible solution to the original problem, the value of the game parameter is bounded by that of the original optimization parameter.
\begin{proposition}\label{prop:trivial_bounds}
Always $b(G) \le \bg(G) \le 2b(G)-1$ and $b(G) \le \widehat \bg(G) \le 2b(G)$.
\end{proposition}

In Section~\ref{sec:pre}, we show that these bounds, though elementary, are best possible in a strong sense: 

\begin{theorem}\label{thm:A}
For every rational number $r$ in $[1,2)$, there exists a graph $G$ such that 
$$
\frac {\bg(G)}{b(G)} = r.
$$
\end{theorem}

We next turn our attention to complete graphs.  Our next main result (proved in Section~\ref{sec:complete}) provides the asymptotic behavior of the game brush number for $K_n$.

\begin{theorem}\label{thm:complete}
$\bg(K_n) = (1+o(1)) n^2/e$
\end{theorem}

Finally, we move to random graphs. The \emph{random graph} $\G(n,p)$ consists of the probability space $(\Omega, \mathcal{F}, \Pr)$, where $\Omega$ is the set of all graphs with vertex set $\{1,2,\dots,n\}$, $\mathcal{F}$ is the family of all subsets of $\Omega$, and for every $G \in \Omega$,
$$
\Pr (G) = p^{|E(G)|} (1-p)^{{n \choose 2} - |E(G)|} \,.
$$
This space may be viewed as the set of outcomes of ${n \choose 2}$ independent coin flips, one for each pair $(u,v)$ of vertices, where the probability of success (that is, adding edge $uv$) is $p.$ Note that $p=p(n)$ may (and usually does) tend to zero as $n$ tends to infinity.  All asymptotics throughout are as $n \rightarrow \infty $ (we emphasize that the notations $o(\cdot)$ and $O(\cdot)$ refer to functions of $n$, not necessarily positive, whose growth is bounded). We say that an event in a probability space holds \emph{asymptotically almost surely} (or \emph{a.a.s.}) if the probability that it holds tends to $1$ as $n$ goes to infinity.

Our main result here is the following:

\begin{theorem}\label{thm:random_brushing}
For $p = p(n) \gg \ln n / n$ and $G \in \G(n,p)$, a.a.s. 
$$
\bg(G) = (1+o(1))p\bg(K_n) = (1+o(1))pn^2/e.
$$
\end{theorem}
\noindent The result is proved in Section~\ref{sec:random}.

\section{Preliminaries and relation to the brush number}\label{sec:pre}

We begin with a formal definition of the brushing game, along with some terminology used throughout the paper.  To facilitate reasoning about the game, we define it in greater generality than was mentioned in the introduction.  A {\em configuration} of a graph $G$ is an assignment of some nonnegative integer number of brushes to each vertex of $G$; we represent a configuration by a map $f : V(G) \rightarrow \N \cup \{0\}$.

\begin{definition}
The {\em brushing game} on a graph $G$ with {\em initial configuration} $f$ has players {\em Max} and {\em Min}.  Initially, all vertices of $G$ are deemed {\em dirty}, and each vertex $v$ contains $f(v)$ brushes.  

The players alternate turns.  At the beginning of each turn, the player whose turn it is adds one brush to any dirty vertex.  After this, if some vertex $v$ has at least as many brushes as dirty neighbors, then $v$ {\em fires}: one brush is added to each dirty neighbor of $v$, and $v$ itself is marked {\em clean}.  (The process of firing $v$ is also referred to as {\em cleaning} $v$.)  Vertices continue to fire, sequentially, until no more vertices may fire, at which point the turn ends.  (If two or more vertices are simultaneously able to fire, then the order of firings is chosen arbitrarily---although we will see in Observation~\ref{obs:basic} that the order does not matter.)  If at this point all vertices of $G$ are clean, then the game ends.

When both players play optimally, the number of brushes placed during the game is the {\em game brush number} of $G$ with initial configuration $f$, denoted $\bg(G;f)$ if Min takes the first turn and by $\widehat \bg(G;f)$ when Max does.  When $f$ is identically zero, we write $\bg(G)$ in place of $\bg(G;f)$ and refer to it as the {\em game brush number of $G$}.  When Min (resp.\ Max) starts the game, we sometimes refer to the process as the {\em Min-start} (resp.\ {\em Max-start}) game.  

In the Min-start (resp.\ Max-start) game, a {\em round} of the game consists of one turn by Min (resp.\ Max) and the subsequent turn by Max (resp.\ Min).
\end{definition}

We say that $G$ can be {\em cleaned by} the configuration $f$ when some list of vertex cleanings, starting from $f$, results in every vertex of $G$ being cleaned.  Given configurations $f$ and $g$ of a graph $G$, we say that $f$ {\em dominates} $g$ provided that $f(v) \ge g(v)$ for all $v \in V(G)$.

The following well-known facts will be of use: 
\begin{observation}[\cite{MNP}]\label{obs:basic}
For any graph $G$,
\begin{enumerate}
\item Let $f$ and $g$ be configurations of $G$ such that $f$ dominates $g$.  If $G$ can be cleaned by $g$, then it can also be cleaned by $f$.
\item If $G$ can be cleaned by $f$, then every maximal list of vertex firings (starting from $f$) cleans all vertices of $G$. 
\end{enumerate}
\end{observation}
Fact (2) above implies that in the brushing game, it is the multiset of brushes placed that determines whether the game has ended; the order of moves is irrelevant.

We also need the following result from~\cite{MNP}.
\begin{theorem}[\cite{MNP}, Theorem 3.1]\label{thm:odd_degree}
The brush number of any graph is at least half the number of vertices of odd degree.
\end{theorem}

\bigskip

Now we are ready to prove Theorem~\ref{thm:at_most_one}.  We obtain the theorem as a simple consequence of a more general lemma.  

\begin{lemma}\label{lem:domination}
Let $f$ and $g$ be configurations of a graph $G$.  If $f$ dominates $g$, then $\bg(G;f) \le \bg(G;g)$ and $\widehat \bg(G;f) \le \widehat \bg(G;g)$.  Moreover, $\bg(G;g) - \bg(G;f) \le 2 \sum_v (f(v) - g(v))$ and $\widehat \bg(G;g) - \widehat \bg(G;f) \le 2 \sum_v (f(v) - g(v))$.
\end{lemma}
\begin{proof}
We prove $\bg(G;f) \le \bg(G;g)$ and $\bg(G;g) - \bg(G;f) \le 2 \sum_v (f(v) - g(v))$; the proofs of the other inequalities are nearly identical.  Consider two instances of the brushing game on $G$: the {\em $f$-game}, in which we have initial configuration $f$, and the {\em $g$-game}, in which we have initial configuration $g$.  Intuitively, the additional brushes initially present in the $f$-game cannot lengthen the game.  However, for a formal proof, more care is needed: these extra brushes could cause some vertices to fire earlier in the $f$-game, which could invalidate some of Min's desired moves (since she cannot add brushes to clean vertices).

We give a strategy for Min in the $f$-game.  Min plays the $f$-game and $g$-game simultaneously; the $f$-game is the ``real'' game in which both players play, while the $g$-game is ``imagined'' by Min to guide her play in the $f$-game.  At all times, we denote by $f^*$ the current configuration of the $f$-game and by $g^*$ the current configuration of the $g$-game. Likewise, for all $v \in V(G)$, at all times $c_f(v)$ and $c_g(v)$ denote the number of clean neighbors of $v$ in the $f$-game and in the $g$-game, respectively.

Min aims to ensure that the $f$-game finishes no later than the $g$-game.  She maintains two invariants: after each turn, 
\begin{description}
\item[(1)] $f^*(v)-c_f(v) \ge g^*(v)-c_g(v)$ for each vertex $v$ that is dirty in both games, and
\item[(2)] every clean vertex in the $g$-game is also clean in the $f$-game.
\end{description}
Note that firing the neighbor of some dirty vertex $v$ increases $f^*(v)$ by 1 and decreases $c_f(v)$ by 1, so invariant (1) is maintained under firing a vertex (in both games).  We claim that invariant (2) follows from invariant (1). We prove this through induction on the number of turns played.  Clearly, invariant (2) holds at the beginning of the game.  Fix some nonnegative integer $t$ and suppose both invariants hold after $t$ turns; we show that invariant (2) must also hold after $t+1$ turns.  Consider the state of both games at the beginning of turn $t+1$, before any vertices have fired.  A clean vertex $v$ may fire in the $f$-game if and only if $f^*(v) \ge \deg(v) - c_f(v)$ or, equivalently, $f^*(v)-c_f(v) \ge \deg(v) - 2c_f(v)$.  Likewise, $v$ may fire in the $g$-game if and only if $g^*(v) \ge \deg(v) - c_g(v)$.  Suppose now that vertices $x_1, x_2, \ldots, x_k$ fire in the $g$-game, in that order.  Consider each $x_i$ in turn.  If $x_i$ is already clean in the $f$-game, then after it fires in the $g$-game, both invariants still hold.  On the other hand, if $x_i$ is dirty in the $g$-game, then 
$$f^*(x) - c_f(x) \ge g^*(x) - c_g(x) \ge \deg(x) - 2c_g(x) \ge \deg(x) - 2c_f(x),$$
where the first inequality follows from invariant (1), the second from the assumption that $x_i$ can fire in the $g$-game, and the last from invariant (2).  Thus $x_i$ may fire in the $f$-game as well, and again both invariants hold.  It follows that invariant (2) holds after all of the $x_i$ have fired and, hence, after turn $t+1$.  Thus we need only explicitly verify that invariant (1) holds throughout the game, since this implies that invariant (2) holds as well.

Since $f$ dominates $g$, invariant (1) holds initially.  Subsequently, on Max's turns, Max plays in the $f$-game, after which Min imagines the same move in the $g$-game.  (Note that every valid move in the $f$-game is valid also in the $g$-game, since every dirty vertex in the $f$-game is also dirty in the $g$-game.)  This clearly maintains invariant (1).  On Min's turns, Min chooses some optimal move in the $g$-game; suppose she adds a brush to vertex $v$.  If $v$ is clean in the $f$-game, then Min plays the same move there, which maintains invariant (1).  Otherwise, Min plays any valid move in the $f$-game.  (If there are no valid moves in the $f$-game, then the $f$-game has ended no later than the $g$-game, as desired.)  In this case, $v$ must have been clean in the $f$-game but dirty in the $g$-game, so $v$ itself has no bearing on the invariant; indeed, since we consider only those vertices that are dirty in both games, invariant (1) is maintained.

Min follows an optimal strategy for the $g$-game, so that game lasts for at most $\bg(G;g)$ turns.  Throughout the game, every clean vertex in the $g$-game is also clean in the $f$-game; hence the $f$-game finishes no later than the $g$-game.  Consequently $\bg(G;f) \le \bg(G;g)$, as claimed.

Finally, to prove $\bg(G;g) - \bg(G;f) \le 2 \sum_v (f(v) - g(v))$, we observe that Min could apply the following strategy: starting from configuration $g$, for each vertex $v$, iteratively add up to $f(v)-g(v)$ brushes to $v$ (as long as $v$ remains dirty).  This requires at most $\sum_v (f(v) - g(v))$ turns by Min (and hence an equal number of turns by Max).  The resulting configuration dominates $f$, so from this point, Min can end the game in at most $\bg(G;f)$ additional turns.
\end{proof}

Lemma~\ref{lem:domination} formalizes the intuition that extra brushes cannot hinder Min nor help Max. As a consequence, when seeking an upper bound on $\bg(G)$ for some graph $G$, we may ``ignore'' brushes that have been placed, pretending that they simply do not exist; this cannot shorten the remainder of the game.  This can be quite useful when we want to give a strategy for Min but do not want to worry about some or all of Max's moves.  Likewise, when we seek a lower bound on $\bg(G)$, we may pretend that the graph contains extra brushes that were never actually played.  Conversely, ignoring a brush can increase the length of the game by at most 2, and adding a brush can decrease the length of the game by at most 2.

Lemma~\ref{lem:domination} also yields a quick proof of Theorem~\ref{thm:at_most_one}:

\begin{proof}[Proof of Theorem~\ref{thm:at_most_one}.]
Let Max play an optimal first move in the Max-start game, and denote the resulting configuration by $f$.  The remainder of the game can be viewed as a Min-start game with initial configuration $f$.  Thus by Lemma~\ref{lem:domination}, $\widehat \bg(G) = \bg(G;f)+1 \le \bg(G)+1$.

Similarly, let Min play an optimal first move in the Min-start game, and denote the resulting configuration by $g$.  This time, Lemma~\ref{lem:domination} yields $\bg(G) = \widehat \bg(G;g)+1 \le \widehat \bg(G)+1$.  This completes the proof.
\end{proof}

\bigskip
Recall from Proposition~\ref{prop:trivial_bounds} that always $b(G) \le \bg(G) \le 2b(G)-1$ and $b(G) \le \widehat \bg(G) \le 2b(G)$, because Min can use her first $b(G)$ turns to play brushes in some configuration realizing $b(G)$.  Thus always $1 \le \bg(G) / b(G) < 2$.  It is natural to ask whether the set of all such ratios, over all connected graphs $G$, is dense in $[1,2)$. 

The constructions we use to answer this question make use of combs and sunlets.  The {\em $n$-comb} $B_n$ and {\em $n$-sunlet} $S_n$ are the graphs obtained from $P_n$ and from $C_n$, respectively, by attaching one pendant leaf to each vertex.  
\begin{lemma}\label{lem:ratio}
Fix $n_1, n_2, \ldots, n_m$, all at least 2.  Let $G$ be the disjoint union of $B_{n_1}, B_{n_2}, \ldots, B_{n_m}$.  Suppose furthermore that, in each component of $G$, the two vertices of degree 2 contain one brush each, while all other vertices are empty.  If we play the brushing game from this initial configuration, then the number of turns needed to clean $G$ is $\sum_i n_i - m$, regardless of which player moves first.
\end{lemma}
\begin{proof}
Let $n = \sum_i n_i$ and let $f$ denote the specified initial configuration of $G$.  We prove that $\bg(G;f)$ and $\widehat \bg(G;f)$ are bounded above and below by $n-m$.

We begin with the lower bound.  The graph $G$ can be obtained from $S_{n+m}$ by placing one brush on each of $m$ appropriately-chosen pendant leaves and their neighbors, then allowing these vertices to fire.  (Here we are supposing for convenience that clean vertices are deleted from the graph; we may do so because clean vertices no longer affect the game.)  Since $S_{n+m}$ contains $2(n+m)$ vertices of odd degree, Theorem~\ref{thm:odd_degree} yields $b(S_{n+m}) \ge n+m$.  Since $2m$ brushes have already been placed, the number of additional brushes needed to clean $G$ is at least $n-m$; this establishes the desired bound.

For the upper bound, we give a strategy for Min.  We use induction on $n$.  When $n=2$ the claim is clear, so suppose $n \ge 3$.  If Min plays first, then she plays on any vertex of degree 2.  If the component in which she played was isomorphic to $B_2$, then the entire component gets cleaned, and only $m-1$ components remain.  By the induction hypothesis, the number of additional turns needed is at most $(n-2) - (m-1)$, which simplifies to $n-m-1$, and the desired bound follows.  If instead the component was isomorphic to $B_k$ for $k \ge 3$, then only the vertex at which Min played and its pendant leaf fire, leaving a component isomorphic to $B_{k-1}$ with the desired initial configuration.  The induction hypothesis again shows that the game lasts at most $n-m-1$ more turns, and the desired bound again follows.

If instead Max plays first, then he has four options: he may play on a vertex of degree 2, on a pendant leaf attached to a vertex of degree 2, on a vertex of degree 3, or on a pendant leaf attached to a vertex of degree 3.  In the first two cases, the desired bound follows as before.  Otherwise, let $v$ be the vertex at which Max plays.  If $v$ has degree 3, then Min plays at its pendant leaf; if $v$ is itself a pendant leaf, then Min plays at its neighbor.  In either case, $v$ and its neighbor both fire, leaving a new graph $H$ of the specified form; say $H$ has $m'$ components and $2n'$ vertices.  Of the two vertices cleaned, let $x$ be the one with degree 3.  If $x$ has two neighbors of degree 2, then $x$ belongs to a copy of $B_3$ in $G$, all vertices of which fire.  Hence $n' = n-3$ and $m' = m-1$, so the number of turns needed to clean $H$ is at most $n'-m'$, which simplifies to $n-m-2$.  If instead $x$ has one neighbor of degree 2, then $x$, its neighbor, and their pendant leaves all fire.  Thus $n' = n-2$ and $m' = m$, so again the number of turns needed to clean $H$ is at most $n-m-2$.  Finally, if $x$ has no neighbors of degree 2, then $x$ and its neighbor fire, leaving two new components in $H$.  Now $n' = n-1$ and $m' = m+1$, so once again the number of turns needed to clean $H$ is at most $n-m-2$.  In any case, the total length of the game on $G$ is at most $2 + (n-m-2)$, which simplifies to $n-m$, as claimed.
\end{proof}

Before proceeding, we need one more technical lemma.

\begin{lemma}\label{lem:subdiv}
Let $f$ and $g$ be configurations of a graph $G$, and let $v$ be a vertex of degree 2 in $G$.  Suppose that $f(v) = 0$, $g(v) = 1$, and $f(u) = g(u)$ for all $u \in V(G) - \{v\}$.  If $G$ can be cleaned by $g$, then it can be cleaned by $f$.  
\end{lemma}
\begin{proof}
When some neighbor of $v$ fires, $v$ loses a dirty neighbor and gains a brush.  Hence, whether we start from $f$ or $g$, vertex $v$ may fire when and only when one of its neighbors has fired.  In both cases, when $v$ fires, it sends one brush to each remaining dirty neighbor.  Thus, after $v$ fires, the number of brushes on each remaining dirty vertex is independent of which configuration we started from.  It follows that any list of vertex firings that cleans $G$ starting from $g$ also cleans $G$ starting from $f$.
\end{proof}

We are now ready to prove Theorem~\ref{thm:A}.

\begin{proof}[Proof of Theorem~\ref{thm:A}.]
For positive integers $k$ and $n$, with $k \le n$, let $G_{n,k}$ be obtained from $S_n$ by choosing any $k$ consecutive pendant edges and subdividing each one $2n$ times.  We show that by choosing $n$ and $k$ appropriately, we can make the ratio $\bg(G_{n,k}) / b(G_{n,k})$ take on any rational value in $[1, 2)$.  Since $G_{n,k}$ has $2n$ vertices of odd degree, Theorem~\ref{thm:odd_degree} yields $b(G_{n,k}) \ge n$.  Conversely, it is easy to see that $n$ brushes suffice to clean $G_{n,k}$: place one brush on each of $n-1$ leaves and one brush on the neighbor of one of these vertices.  Thus, $b(G_{n,k}) = n$.  We claim that $\bg(G_{n,k}) = n+k-1$, from which it would follow that whenever $q \le p < 2q$, we have $\bg(G_{q,p-q+1}) = (q+(p-q))/q = p/q$.

To show that $\bg(G_{n,k}) \ge n+k-1$, we give a strategy for Max.  Call the pendant paths of length $n$ {\em threads}, and call the vertices of degree 2 {\em thread vertices}.  
For as many turns as he can, Max places a brush on any thread vertex that does not already have one.  Since $b(G_{n,k}) = n$, we may suppose that the game lasts no more than $2n-1$ turns.  Since each thread has $2n$ thread vertices, Max can adhere to this strategy until all of the threads have been cleaned; in particular, he never plays more than one brush on any thread vertex.  We may also suppose without loss of generality that Min played no brushes on thread vertices, since she could have achieved the same result by playing a brush that thread's leaf.  Thus, by Lemma~\ref{lem:subdiv}, we may ``ignore'' the brushes on thread vertices, in the sense that they have no impact on the duration of the game.  Once all the threads have been cleaned, Max plays arbitrarily.  Consider the state of the game just after the last thread is cleaned, and let $G'$ be the graph induced by the clean edges of $G$.  Each thread contains at least two vertices of odd degree in $G'$: the leaf necessarily has degree 1, and the other endpoint has either degree 1 or degree 3, depending on whether it has fired.  Since $G'$ has $2k$ vertices of odd degree, and since Max's brushes did not contribute to its cleaning, Min must have played at least $k$ brushes so far.  Consequently, Max plays at least $k-1$ brushes on thread vertices.  Since $b(G_{n,k}) = n$, and since Max plays at least $k-1$ brushes that do not contribute to the cleaning of the graph, at least $n+k-1$ brushes are played throughout the game.

For the upper bound, we give a strategy for Min.  While there are still threads with dirty vertices, Min places a brush on some such thread's leaf, cleaning the whole thread.  This happens at most $k$ times.  By Lemma~\ref{lem:domination}, since we seek an upper bound on the length of the game, we may ignore the brushes placed by Max during this time.  At this point, the graph consists of an $(n-k)$-sunlet in which one edge has been subdivided $k$ times, and in which each vertex of degree 2 contains one brush.  If it is Min's turn, then she plays on a vertex of degree 2.  This causes all vertices of degree 2 to fire, and produces a copy of $B_{n-k}$ having the initial configuration described in Lemma~\ref{lem:ratio}.  Since at most $2k$ brushes have already been played, and Lemma~\ref{lem:ratio} states that the remainder of the game lasts exactly $n-k-1$ turns, the length of the game is at most $2k + (n-k-1) = n+k-1$, as desired.  Likewise, we obtain the same bound if it is Max's turn, and he chooses to play on a vertex of degree 2.  

If it is Max's turn and he chooses to play on a leaf, then Min plays on its neighbor; if Max chooses to play on a vertex of degree 3, then Min plays on its pendant leaf.  In either case, we have a graph of the form specified in Lemma~\ref{lem:ratio}, except that one component contains the subdivided edge.  We claim that, despite these extra vertices, the game still lasts for at most $n-k-1$ more turns.  The vertices on the subdivided edge must eventually be cleaned in one of two ways: either by cleaning one of the incident vertices of degree 3, or by placing an extra brush directly on one of the vertices of degree 2.  In both cases, once one of these vertices fires, the rest follow.  In the former case, the subdivided edge had no impact on the game.  In the latter case, the resulting graph has exactly the form required by Lemma~\ref{lem:ratio}; in cleaning the subdivided edge, one turn has been spent, and the number of components has increased by 1.  Either way, the strategies from Lemma~\ref{lem:ratio} still show that at most $n-k-1$ more turns are needed.  Hence the length of the game is at most $2k+(n-k-1) = n+k-1$, as claimed.  This completes the proof.
\end{proof}

\section{Complete graphs}\label{sec:complete}

We next turn our attention to complete graphs.  To simplify the presentation of our main result, we first present both players' optimal strategies, we next analyze the performance of these strategies, and we save the proof of optimality for last.  We give two strategies, one for Max and one for Min.
\begin{itemize}
\item The {\em balanced strategy} for Max: on each turn, Max adds a brush to the dirty vertex having the fewest brushes.
\item The {\em greedy strategy} for Min: Min selects the dirty vertex having the most brushes, adds brushes to that vertex until it gets cleaned, and repeats.
\end{itemize}

As we will show later, these strategies yield a Nash equilibrium---that is, the balanced strategy is an optimal response to the greedy strategy and vice-versa.  Hence the game brush number of $K_n$ is precisely the length of the game when the strategies are employed against each other.  While there is no simple formula for $\bg(K_n)$, we can use differential equations to determine the asymptotics.

\begin{lemma}\label{lem:complete_DE}
Consider the brushing game on $K_n$.  When Max uses the balanced strategy and Min uses the greedy strategy, the game ends in $(1+o(1)) n^2/e$ turns.
\end{lemma}
\begin{proof}
At all times throughout the game, let $x$ denote the number of vertices cleaned so far and $y$ the number of brushes placed.  We define the $k$th {\em phase} of the game to be the sequence of turns during which there are exactly $k$ clean vertices, starting with the turn after that on which the $k$th vertex fires and ending on the turn on which the $(k+1)$st vertex fires.

Consider the state of the game at the beginning of phase $x$ (and suppose the game has not yet finished). We analyze the length of this phase. With $x$ vertices cleaned, each remaining vertex must accumulate $n-x-1$ brushes before firing.  The strategies employed by both players ensure that no brushes have been ``left behind'' on clean vertices and that each of the $n-x$ remaining dirty vertices has the same number of brushes (up to a difference of 1).  Hence each dirty vertex has $\frac{y}{n-x} + O(1)$ brushes.  Min must therefore place another $n-x-\frac{y}{n-x} + O(1)$ brushes before the next vertex fires; since Max also plays, $2(n-x-\frac{y}{n-x}) + O(1)$ brushes are introduced in this phase.  

To investigate the asymptotic behavior of the game, we simplify the analysis by modeling the brushing game (a discrete process) using differential equations (a continuous model).  The argument above yields
$$y(x+1) - y(x) = 2 \left( n-x-\frac{y}{n-x} \right) + O(1).$$
We next normalize both parameters by letting $t = x/n$ and $f(t)=y/n^2$.  We thus arrive at the differential equation
$$f'(t) = 2\left (1-t-\frac{f(t)}{1-t}\right ), \quad f(0) = 0.$$
It is easily verified that the solution to this differential equation is $f(t) = -2(1-t)^2 \ln (1-t)$. The game ends during phase $(1+o(1))t_0n$, where $t_0$ is such that $f'(t_0) = 0$.  Solving for $t_0$ yields $t_0 = 1-e^{-1/2}$ and $f(t_0)=1/e$. Hence we expect that the number of brushes played throughout the game is $(1+o(1)) f(t_0) n^2 = (1+o(1)) n^2/e$.

For a formal proof, we now return to the original (discrete) model of the brushing game.  Fix a nonnegative integer $x$, and suppose that the game does not end before phase $x$.  As argued above,
$$2 \left( n - x - \frac {y(x)}{n-x} \right) - C \le y(x+1)-y(x) \le 2 \left( n - x - \frac {y(x)}{n-x} \right) + C$$
for some constant $C$.  Intuitively, since the ``true'' length of this phase is close to its value under the differential equation model, the ``true'' length of the game should likewise be close to the value suggested by the solution of the differential equation.  

We can prove this formally through induction on $x$.  For all nonnegative integers $x$ such that the game does not end before phase $x$, we claim that 
$$
y(x) = 2 (n-x)^2 \ln \left( \frac {1}{1-x/n} \right) + O(x).
$$
More precisely, we prove that $y(x) \le 2 (n-x)^2 \ln \left( \frac {1}{1-x/n} \right) + C'x$ for some constant $C'$; the proof of the analogous lower bound is similar.  The desired bound trivially holds when $x = 0$.  For the inductive step, fix $x \ge 0$, suppose that the game does not end before phase $x$, and assume that $y(x) \le 2 (n-x)^2 \ln \left( \frac {1}{1-x/n} \right) + C'x$.  Note that the game must end before phase $n/2$: once $n/2$ vertices have fired, the remaining vertices each have at least $n/2$ brushes and at most $n/2-1$ clean neighbors.  Hence we may suppose $x \le n/2$.  Now
\begin{eqnarray*}
y(x+1) &=& y(x) + \Big( y(x+1)-y(x) \Big)\\
&\le& 2 (n-x)^2 \ln \left( \frac {1}{1-x/n} \right) + C'x + 2 \left( n - x - \frac {y(x)}{n-x} \right) + C \\
&\le& 2 (n-x) \left( (n-x) \ln \left( \frac {1}{1-x/n} \right) + 1 - 2 \ln \left( \frac {1}{1-x/n} \right) \right) - \frac {2C'x}{n-x} + C'x + C \\
&=& 2 (n-x) \left( (n-x-2) \ln \left( \frac {1}{1-x/n} \right) + 1 \right) + C'x \left( 1 - \frac {2}{n-x} \right)+ C \\
&\le& 2 (n-x) \left( (n-x-2) \ln \left( \frac {1}{1-(x+1)/n} \right) + O \left( \frac {1}{n-x} \right) \right) + C'x + C \\
&=& 2 (n-x) (n-x-2) \ln \left( \frac {1}{1-(x+1)/n} \right) + C'x + O(1) \\
&=& 2 \Big( (n-x-1)^2+1 \Big) \ln \left( \frac {1}{1-(x+1)/n} \right) + C'x + O(1) \\
&=& 2 (n-(x+1))^2 \ln \left( \frac {1}{1-(x+1)/n} \right) + C'x + O(1),
\end{eqnarray*}
where the first and second inequalities follow from the induction hypothesis and the last line follows from the assumption that $x \le n/2$ (hence $\ln (1/(1-(x+1)/n)) = O(1)$).  If $C'$ is taken to be sufficiently large, then $C'x + O(1) \le C'(x+1)$, so the claimed bound holds.  A similar argument establishes the analogous lower bound on $y(x)$.  As with the differential equation model, the game ends immediately following phase $t_0$, where $t_0$ is the largest nonnegative integer such that $y(t_0+1) - y(t_0) > 0$.  Solving for $t_0$ yields $t_0 = (1+o(1))(1-e^{-1/2})n$, from which it follows that that $y(t_0+1) = (1+o(1)) n^2/e$.  Since exactly $y(t_0+1)$ brushes are played throughout the game, this completes the proof. 
\end{proof}

We now show that the strategies analyzed in Lemma~\ref{lem:complete_DE} are optimal for both players.  In this analysis, we focus on the number of brushes placed directly on each vertex (as opposed to those obtained from clean neighbors).  Moreover, we ``sort'' the vertices by the numbers of brushes played, in non-increasing order: always $v_1$ denotes the vertex having received the most brushes, $v_2$ the vertex having received the next most, and so on.  By the symmetry of $K_n$, we may suppose that moves are always played so as to avoid re-indexing of the vertices.  For example, if $v_i$ and $v_{i+1}$ have received the same number of brushes, then placing a brush on $v_{i+1}$ would force re-indexing: the old $v_i$ becomes the new $v_{i+1}$, and the old $v_{i+1}$ becomes the new $v_i$.  However, we may just as well suppose that the brush was placed directly on $v_i$, since this produces an isomorphic configuration. This assumption is equivalent to forbidding either player from adding a $k$th brush to $v_i$ until $k$ brushes have been played on $v_{i-1}$.  Since we track only the number of vertices played directly on each vertex, firing a vertex has no effect on the indexing.

Using these conventions, in the course of the game on $K_n$, we clean $v_1$ first, then $v_2$, and so on.  For $i < n/2$, vertex $v_i$ receives $i-1$ brushes from earlier neighbors and needs $n-i$ brushes to fire.  Hence $v_i$ fires when and only when the following two conditions have been met: first, vertices $v_1, v_2, \ldots, v_{i-1}$ are already clean; second, at least $n-2i+1$ brushes have been played on $v_i$.  Thus the game ends when and only when $n-2i+1$ brushes have been played on $v_i$ for all $i$ in $\{1, 2, \ldots, \floor{n/2}\}$.  Viewing this condition graphically yields a ``triangle'' of brushes that must be placed before the game ends (see Figure~\ref{fig:critical_triangle}).  We refer to this triangle as the {\em critical triangle}, to brushes placed within the triangle as {\em in-brushes}, and to brushes placed outside as {\em out-brushes}.  (Formally, the $k$th brush played on vertex $v_i$ is an in-brush if $k \le n-2i+1$ and an out-brush otherwise.)  Since the number of in-brushes played throughout the game is fixed, Max aims to force many out-brushes to be played, while Min aims to prevent this.

\begin{figure}[htbp]
%
%
%
%
%
\centering
\includegraphics{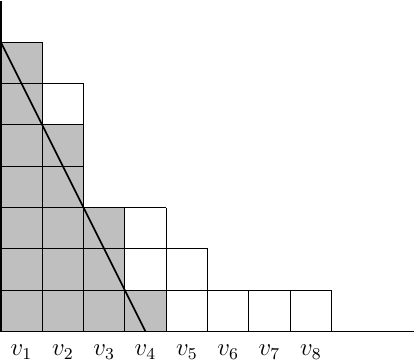}
\caption{An illustration of the critical triangle for $K_8$.  (In-brushes are shaded.)}
\label{fig:critical_triangle}
\end{figure}

Now we are ready to prove Theorem~\ref{thm:complete}.  Since the upper bound and lower bound use substantially different arguments, we split the proof into two parts.  We begin with the lower bound.

\begin{lemma}\label{lem:thm_complete_lower}
$\bg(K_n) \ge (1+o(1)) n^2/e$
\end{lemma}

\begin{proof}
We show that when Max uses the balanced strategy, Min's optimal response is to play the greedy strategy.  Consequently, by Lemma~\ref{lem:complete_DE}, Max can enforce the claimed lower bound by using the balanced strategy.  

Fix $n$ and consider the brushing game on $K_n$ (where Max always uses the balanced strategy).  Call a strategy for Min {\em optimal} if no strategy ends the brushing game on $K_n$ in fewer turns.  Since Max's strategy is fixed, a strategy for Min can be viewed as the sequence of moves she plays throughout the game.  Call a move by Min {\em greedy} if it places a brush on the least-indexed dirty vertex and {\em non-greedy} otherwise.  It suffices to prove that the greedy strategy is optimal.

Fix an optimal strategy $A$ for Min.  If in fact $A$ is the greedy strategy, then there is nothing to prove.  Suppose instead that $A$ contains at least one non-greedy move, and suppose that the last such move is played on Min's $k$th turn.  We show how to transform $A$ into a new optimal strategy, $B$, that either has fewer non-greedy moves than $A$ or has the same number of non-greedy moves but with the last such move happening after Min's $k$th turn.  Iterating this transformation must eventually yield an optimal strategy with fewer non-greedy moves than $A$, and continuing to iterate must eventually yield the greedy strategy (that is, the unique strategy having no non-greedy moves).  Since the transformation always produces an optimal strategy, the greedy strategy must be optimal, as claimed.  

Before giving the construction of strategy $B$, we observe that Min's last move of the game must be greedy.  If Min's last move concludes the game, then with that move Min plays the final in-brush.  Consequently, on Min's last turn, only one in-brush remains to be played.  Under the greedy strategy, Min always plays an in-brush; since Min's final move places the last in-brush, this move must in fact be greedy.  Suppose instead that the game ends on one of Max's turns.  In this case, Max plays the final in-brush; say he places this brush on vertex $v_i$.  By optimality of $A$, Min did not an out-brush on her last turn, since otherwise she could have ended the game by playing the final in-brush.  Hence Min played an in-brush, say on vertex $v_j$.  Note that after Min's last turn, exactly $n-2i$ brushes have been played on $v_i$, exactly $n-2j+1$ have been played on $v_j$, and at most $n-2j$ have been played on $v_{j+1}$ (we must have $j+1 \le n$ since no in-brushes can ever be played on $v_n$).  It follows that $j < i$: otherwise, on his ensuing turn, Max would play another brush on $v_{j+1}$ rather than on $v_i$.  Thus, out of the two in-brushes remaining on Min's final turn, she played the one on the lower-indexed vertex, so her move was in fact greedy.

We now explain how to construct strategy $B$.  Define the {\em $A$-game} (resp.\ {\em $B$-game}) to be the instance of the game when Min follows strategy $A$ (resp. strategy $B$).  In the $B$-game, Min plays her first $k-1$ turns exactly as in the $A$-game.  Since Min's last turn in the $A$-game is greedy, she must play at least $k+1$ moves in that game.  Let us consider Min's $k$th and $(k+1)$st turns.  Say Min plays on vertex $v_i$ on her $k$th turn the $A$-game, and say that Max responds by playing on vertex $v_j$.  For her $k$th turn in the $B$-game, Min plays a greedy move, say on vertex $v_{i'}$; say that Max subsequently plays on vertex $v_{j'}$.  By assumption, $i \not = i'$.  

Suppose first that $j = j'$.  On Min's $k$th turn in the $A$-game, the greedy strategy dictates playing on $v_{i'}$, hence $v_{i'}$ is the least-indexed dirty vertex.  Neither Min's nor Max's $k$th move adds any brushes to $v_{i'}$, so it remains the least-indexed dirty vertex on Min's $(k+1)$st turn.  By assumption, in the $A$-game, Min plays greedily after turn $k$; thus, Min plays her $(k+1)$st turn on $v_{i'}$.  In the $B$-game, Min plays on $v_i$ for her $(k+1)$st turn and plays greedily thereafter.  After Min's $(k+1)$st turn the two games have identical configurations, proceed identically, and finish simultaneously.  Thus, strategy $B$ must be optimal.  Moreover, either Min plays fewer non-greedy moves than under strategy $A$ (if $v_i$ turns out to be a greedy move in the $B$-game) or she plays the same number of non-greedy moves but with the last one occurring on her $(k+1)$st turn (if $v_i$ is non-greedy).  

Now suppose $j \not = j'$.  Since Max uses the balanced strategy in both games, $j \not = j'$ implies $i = j'$.  In other words, the balanced strategy tells Max to play his $k$th turn on $v_{j'}$ unless Min plays there first.  Moreover, since Max uses the balanced strategy, if he ever causes a vertex to fire, subsequently all remaining vertices fire and the game ends.  If $j = i'$, then after Max's $k$th turn, both games have the same configuration.  Min henceforth plays greedily in the $B$-game and, as above, it follows that $B$ is an optimal strategy of the desired form.  Thus, suppose $j \not = i'$.  Since we have assumed that Min plays at least $k+1$ turns under the optimal strategy $A$, Max's $k$th turn must not cause any vertices to fire.  Hence $v_{i'}$ remains the greedy move on Min's $(k+1)$st turn in the $A$-game, so by assumption she plays there.  In the $B$-game, Min plays her $(k+1)$st move on $v_j$ and plays greedily thereafter: once again, both games have the same configuration, and the claim follows.  This completes the proof.  
\end{proof}

Our proof of the upper bound of Theorem~\ref{thm:complete} is more technical.  For this proof, we consider a generalization of the original game.  In our abstraction of the brushing game on $K_n$, players alternately place brushes on a game board, with the game ending once the ``critical triangle'' is full.  In particular, the board we have been considering has $n$ columns, and the critical triangle has height $n-1$.  To prove the upper bound, we consider the family of games $\mathcal{BG}(w,h,t)$.  Each such game is similar to the original, except that the game board has ``width'' $w$, the critical triangle has height $h$, and the game begins with $t$ consecutive turns by Max.  

\begin{definition}
The game $\mathcal{BG}(w,h,t)$ has two players, {\em Max} and {\em Min}, and is played on a rectangular game board divided into square {\em cells}.  The board consists of $w$ {\em columns} and $h$ {\em rows} of cells.  Initially, each cell of the board is deemed {\em empty}, and each column is deemed {\em dirty}.  

The players take turns filling cells of the game board.  On each turn, a player may fill the cell in row $i$, column $j$ provided that:
\begin{itemize}
\item column $j$ is still dirty, 
\item either $i=0$ or the cell in row $i-1$, column $j$ has already been filled, and
\item either $j = 1$ or the cell in row $i$, column $j-1$ has already been filled.  
\end{itemize}
When a player fills a cell in some column, we say that he or she {\em plays in} that column.  After each turn, each dirty column $k$ is deemed {\em clean} if that column has at least $h-2k+2$ filled cells, and either $k=1$ or column $k-1$ is already clean.  When a column changes from dirty to clean, we also say that the column {\em fires}.  Columns continue to fire until no more columns are able to fire.  Once all columns have fired, the game ends.

At the beginning of the game, Max takes $t$ consecutive turns.  After this, the players alternate turns, with Min taking the $(t+1)$st turn, Max taking the $(t+2)$nd, and so on.  Play continues in this manner until the game ends.  The $r$th {\em round} of the game consists of turns $t+2r-1$ and $t+2r$ -- that is, Min's $r$th turn and Max's subsequent turn. The {\em length} of the game is the number of turns taken throughout the course of the game.  Min aims to minimize the length of the game, while Max aims to maximize it.
\end{definition}

Note that $\mathcal{BG}(n,n-1,0)$ is equivalent to the original brushing game on $K_n$.  To upper-bound $\bg(K_n)$, we use the following strategy.  As in the proof of the lower bound, we play two instances of the brushing game: the ``real game'' and the ``ideal game''.  We play some number of turns in both games and view the remainders as as ``sub-games'' isomorphic to instances of $\mathcal{BG}(w,h,t)$ for appropriate choices of $w$, $h$, and $t$.  We then use induction to bound the lengths of the sub-games.  In what follows, we write ``$\bg^*(w,h,t)$'' to refer to the length of $\mathcal{BG}(w,h,t)$ when Min uses the greedy strategy and Max plays optimally.

\begin{lemma}\label{lem:thm_complete_upper}
$\bg(K_n) \le (1+o(1)) n^2/e$
\end{lemma}

\begin{proof}
It suffices to show that when Min uses the greedy strategy, Max's optimal response is to play the balanced strategy; consequently, by Lemma~\ref{lem:complete_DE}, Min can enforce the claimed upper bound by using the greedy strategy.  In fact, we prove a stronger claim: for any natural numbers $w, h$, and $t$ with $w > h$, the balanced strategy is optimal for Max in $\mathcal{BG}(w,h,t)$, given that Min uses the greedy strategy.  The original claim then follows by taking $w = n$, $h = n-1$, and $t = 0$.  We prove this stronger claim through induction on $w$.  When $w \le 2$, the claim is clear by inspection.

Consider $\mathcal{BG}(w,h,t)$.  As before, we compare two instances of the game.  In the {\em real game}, Max uses any fixed strategy, while in the {\em ideal game}, he uses the balanced strategy.  In both games, Min uses the greedy strategy.  We claim that the ideal game finishes no sooner than the real game, which would establish optimality of Max's balanced strategy.  

If Max is forced to fill the critical triangle within his initial $t$ turns, then the balanced strategy clearly maximizes the length of the game, so we may suppose the game continues past Max's initial turns.  Suppose column 1 fires in round $r$ of the ideal game, and consider the state of both games after $r$ rounds.  At this point, column 1 must have fired in the real game as well, since the balanced strategy minimizes the number of cells filled in column 1 by Max---in fact, when Min uses the greedy strategy, Max fills no cells in that column after his initial $t$ turns.  Let $t' = t+2r-h$; in both games, exactly $t'$ cells have been filled in columns $2, 3, \ldots, w$.  All $t'$ of these cells must have been filled by Max in accordance with the balanced strategy, since Min's greedy strategy ensures that she played her first $r$ moves in column 1.  If the real game has finished after round $r$, then the claim holds.  If instead the real game has not yet finished, then the ideal game also cannot have finished and, moreover, column 2 cannot have fired in that game. 

Suppose first that column 2 has not yet fired in the real game.  Since $t'$ cells have been filled in columns 2 and higher, we may view the remainder of the real game as an instance of $\mathcal{BG}(w-1,h-2,t')$ in which Min uses the greedy strategy.  In the real game, some of the cells in columns 2 and higher were filled by Min; in this new instance, we pretend that these cells were all filled by Max during his initial $t'$ turns.  The length of this game, and hence the number of turns remaining in the real game, is at most $\bg^*(w-1,h-2,t')-t'$.  We may likewise view the remainder of the ideal game as an instance of $\mathcal{BG}(w-1,h-2,t')$ in which Min again uses the greedy strategy.  However, this time, all cells already filled actually were filled by Max and, moreover, they were filled in accordance with the balanced strategy.  By the induction hypothesis, the number of turns remaining in the ideal game is exactly $\bg^*(w-1,h-2,t')-t'$, so the ideal game lasts at least as long as the real game, as claimed.

Now suppose instead that after $r$ rounds, exactly $k$ columns have fired in the real game for some $k \ge 2$.  If $k = w$, then the real game is complete and the claim follows, so suppose otherwise.  This time, we cannot necessarily view the remainder of the real game as an instance of $\mathcal{BG}(w-1,h-2,t')$, since it might be that more than $h-2$ cells were filled in column $2$ (which would be impossible in such a game).  Instead, we consider only columns $k+1, k+2, \ldots, w$.  Letting $x$ denote the total number of moves played in columns $2, 3, \ldots, k$, the number played in columns $k+1, k+2, \ldots, w$ is $t'-x$.  Since column $k+1$ has not yet fired, it has fewer than $h-2k$ filled cells.  Hence we may view the remainder of the real game as an instance of $\mathcal{BG}(w-k,h-2k,t'-x)$ in which Min uses the greedy strategy.  

Let $T_r$ and $T_i$ denote the number of turns remaining in the real game and the ideal game, respectively.  We aim to show that $T_r \le T_i$, from which the claim would follow.  We begin by upper-bounding $T_r$.  Just as before, we may pretend that Max played all $t'-x$ moves in this new instance.  By the induction hypothesis, we may suppose (since we seek an upper bound on $T_r$) that these moves were all played in accordance with the balanced strategy.  This means that Max has filled $y_r$ rows and $z_r$ cells in the next row, where $y_r = \floor{\frac{t'-x}{w-k}}$ let $z_r = t'-x - y_r(w-k)$.  Thus $T_r \le \bg^*(w-k,h_r,z_r)-z_r$, where $h_r = h-2k-y_r$.

Let us now analyze $T_i$.  In the ideal game, exactly $t'$ moves have been played in columns $2, 3, \ldots, w$, all played by Max in accordance with the balanced strategy.  Let $y_i = \floor{\frac{t'}{w-1}}$ and let $z_i = t'-y_i(w-1)$.  Viewing the remainder of the ideal game as a sub-game on a playfield of width $w-1$ and height $h-2-y_i$, we see that $T_i = \bg^*(w-1,h_i,z_i)-z_i$, where $h_i = h-2-y_i$.  

For the sake of comparing the two games, we claim that $y_i \le y_r + \frac{x-y_r(k-1)}{w-1} + 1$.  There are $x-y_r(k-1)$ filled cells in rows $y_r+1$ and higher of columns $2, 3, \ldots, k$ in the real game; if these moves had instead been played according to the balanced strategy, then they would have completed an additional $\floor{\frac{x-y_r(k-1)}{w-1}}$ full rows, and might also have completed a row that was already partially full.  Thus
$$y_i \le y_r + \frac{x-y_r(k-1)}{w-1} + 1 \le y_r + \frac{(h-1-y_r)(k-1)}{w-1} + 1 \le y_r + \frac{(h-1)(k-1)}{w-1}+1 < y_r + k,$$
where the second inequality uses the observation that no more than $h-1$ moves can be played in column 2 of the real game (hence $x \le (h-1)(k-1)$) and the final inequality uses the assumption that $h < w$.  It now follows that
$$h_i = h-2-y_i > h-2-(y_r+k) = h-(2+k)-y_r \ge h-2k-y_r = h_r.$$

As argued above, $T_i = \bg^*(w-1,h_i,z_i)-z_i$ and $T_r \le \bg^*(w-k,h_r,z_r)-z_r$.  Recall that we aim to show that $T_i \ge T_i$.  Toward this end, first note that adding $k-1$ additional columns to the latter game cannot decrease the length: Min uses the greedy strategy and hence never plays in the new columns, Max can choose to ignore them if he desires.  Invoking optimality of the balanced strategy, we obtain $\bg^*(w-k,h_r,z_r)-z_r \le \bg^*(w-1,h_r,z_r)-z_r$.  Next, we adjust the height of this game.  By the induction hypothesis, the greedy strategy is optimal for Max in $\bg(w-1,h_i,t^*)$ for all $t^*$, hence
$$\bg^*(w-1,h_r,z_r)-z_r = \bg^*(w-1,h_i,z_r+(w-1)(h_i-h_r)) - (z_r+(w-1)(h_i-h_r)),$$
since Max's first $(w-1)(h_i-h_r)$ moves in the latter game fill the first $h_i-h_r$ rows, which yields a configuration equivalent to the former game.  Letting $z^*_r = z_r+(w-1)(h_i-h_r)$ now yields
$$T_r \le \bg^*(w-k,h_r,z_r)-z_r \le \bg^*(w-1,h_r,z_r)-z_r = \bg^*(w-1,h_i,z^*_r)-z^*_r.$$
Finally, it is easily seen that $\bg^*(w-1,h_i,z_i)-z_i \le \bg^*(w-1,h_i,z^*_r)-z^*_r$: the left side of the equation counts the number of turns remaining in the game after an initial $z_i$ greedy moves by Max, while the right side counts the number of turns remaining after $z^*_r$ greedy moves, and the presence of the additional $z_r^*-z_i$ chips cannot lengthen the remainder of the game.  Hence $T_i \ge T_r$, which completes the proof.
\end{proof}

Theorem~\ref{thm:complete} now follows from Lemmas~\ref{lem:thm_complete_lower} and~\ref{lem:thm_complete_upper}.

\section{Random Graphs}\label{sec:random}

We next seek to establish a connection between the brushing game on the complete graph and the game on the random graph.  As an intermediate step, we introduce a fractional variant of the brushing game.  

The fractional brushing game behaves very similarly to the ordinary brushing game, but with two important changes.  First, vertices can contain non-integral numbers of brushes.  Second, a vertex $v$ fires when the number of brushes on $v$ is at least $p$ times the number of dirty neighbors, at which time it sends $p$ brushes to each dirty neighbor (where $p$ is a fixed constant between 0 and 1).

\begin{definition}
Given $p \in (0,1)$, the {\em fractional brushing game} on a graph $G$ has players {\em Max} and {\em Min}.  Throughout the game, each vertex of $G$ contains some nonnegative real number of {\em brushes}.  Initially, all vertices of $G$ are deemed {\em dirty}, and no vertex contains any brushes.  

The players alternate turns.  At the beginning of each turn, if the number of brushes on some vertex $v$ is at least $p$ times the number of dirty neighbors, then $v$ {\em fires}: $p$ brushes are added to each dirty neighbor of $v$, and $v$ itself is marked {\em clean}.  (The process of firing $v$ is also referred to as {\em cleaning} $v$.)  Vertices continue to fire, sequentially, until no more vertices may fire.  If at this point all vertices of $G$ are clean, then the game ends.  Otherwise, the player whose turn it is adds one brush to any dirty vertex, and the turn ends.

Min aims to minimize, and Max aims to maximize, the number of turns taken before the game ends.  When both players play optimally, the total number of turns taken is the {\em fractional game brush number} of $G$ with parameter $p$, denoted $\bg_p(G)$ if Min takes the first turn and by $\widehat \bg_p(G)$ when Max does.  When Min (resp.\ Max) starts the game, we sometimes refer to the process as the {\em Min-start} (resp.\ {\em Max-start}) game.  

In the Min-start (resp.\ Max-start) game, a {\em round} of the game consists of one turn by Min (resp.\ Max) and the subsequent turn by Max (resp.\ Min).
\end{definition}

We next show that the ordinary and fractional game brush numbers are closely connected.

\begin{theorem}\label{thm:fractional}
For every graph $G$ and for every positive $p = p(n)$, 
$$
\bg_p(G) = (1+o(1))p\bg(G) + O(n \ln n).
$$
\end{theorem}
\begin{proof}
Initially, we assume that $1/p$ is always integral.  After proving the theorem under this additional assumption, we briefly explain the modifications needed to extend the argument to all values of $p$.

We bound $\bg_p(G)$ both above and below in terms of $\bg(G)$.  For the lower bound we give a strategy for Max, and for the upper bound we give a strategy for Min.  Since these strategies are quite similar, we present them simultaneously.  Denote the players by ``A'' and ``B'' (where A and B could represent either one of Min or Max).  Player A imagines an instance of the ordinary game on $G$ and uses an optimal strategy in that game to guide his or her play in the fractional game.  The games proceed simultaneously, except that for every round played in the fractional game, player A simulates $1/p$ rounds in the ordinary game.  To simplify the analysis, we introduce a third entity, the {\em Oracle}.  At various points, the Oracle will ``pause'' the games and add extra brushes to various vertices, in an attempt to ``synchronize'' the games by ensuring that all clean vertices in the fractional game are also clean in the ordinary game and vice-versa. Note that each brush introduced by the Oracle can decrease the length of the corresponding game by at most 2.  (This follows from Lemma~\ref{lem:domination}; while the lemma only directly applies to the ordinary game, the same argument suffices for the fractional game with only cosmetic changes.)  Moreover, note that no dirty vertex ever contains brushes added by the Oracle.

Before presenting the strategies, we introduce some additional notation.  For a vertex $v$, let $x^A(v)$ (resp. $y^A(v)$) denote the number of brushes played on $v$ by A in the ordinary game (resp. fractional game); define $x^B(v)$ and $y^B(v)$ similarly.  The {\em discrepancy for A} at a vertex $v$ is defined by $d^A(v) = px^A(v) - y^A(v)$, and the {\em discrepancy for B} is defined by $d^B(v) = px^B(v) - y^B(v)$.  

Player A plays as follows.  By Theorem~\ref{thm:at_most_one}, we may assume without affecting the asymptotics that B plays first (in both games), so each round of the fractional game consists of a move by B and the subsequent move by A.  After B plays in the fractional game, A simulates 
$1/p$ moves by B in the ordinary game; for each move, the imagined B plays on some dirty vertex $v$ minimizing $d^B(v)$.  Player A responds to each imagined move according to some optimal strategy for the ordinary game.  
Now A plays, in the fractional game, on some dirty vertex $v$ maximizing $d^A(v)$.  Finally, if any vertices are clean in the fractional game but not in the ordinary game, then the Oracle adds brushes to these vertices in the ordinary game in order to clean them.  (Note that to clean these vertices, it suffices to add at most $\frac{1}{p}\max\{-d^A(v) - d^B(v),0\}$ brushes to every such vertex $v$.)  Likewise, if any vertices are clean in the ordinary game but not in the fractional game, then the Oracle adds brushes in the fractional game to clean them.  (This time, the Oracle must add at most $\max\{d^A(v)+d^B(v),0\}$ brushes to each such vertex $v$.)  The Oracle repeats these steps until every vertex that is clean in one game is also clean in the other.

Each brush added by the Oracle changes the length of the corresponding game by at most a constant number of turns.  We want to show that the Oracle need not add many brushes.  To this end, we claim that for every vertex $v$, at all points during the game, $d^A(v) = O(\ln n)$.  Let $D_1 = \sum_w \max\{d^A(w), 0\}$, where the summation is over all vertices $w$ that are dirty in the ordinary game.  In any given round, A's plays in the ordinary game may increase $D_1$ by as much as 1.  If at this point $D_1$ exceeds $n$, then $d^A(v) \ge 1$ for some vertex $v$, so A's subsequent move in the fractional game will reduce $D_1$ by $1$.  Consequently, we have $D_1 < n$ at the beginning of each round, so $D_1 < n+1$ at all points during the game.  It follows that, at all times, the number of vertices $v$ for which $d^A(v) \ge 2.1$ must be less than $\frac{n+1}{2.1}$, which in turn is less than $n/2$ for sufficiently large $n$.  

Now let $D_2 = \sum_w \left (\max\{d^A(w), 2.1\}-2.1\right )$, where again the sum is over all vertices $w$ that are dirty in the ordinary game.  As above, we must have $D_2 < m+1$ at all times, where $m$ denotes the number of vertices with discrepancy at least $2.1$.  The argument above shows that $m < \frac{n}{2}$, so always $D_2 < \frac{n+2}{2}$.  Thus, the number of vertices with discrepancy at least $4.2$ -- that is, the number of vertices contributing at least $2.1$ to $D_2$ -- is always at most $\frac{n+2}{2 \cdot 2.1}$, which is less than $n/4$ for sufficiently large $n$.  Iterating this argument shows that for all positive integers $k$, at any point in the game, at most $n/2^k$ vertices have discrepancy at least $2.1k$.  In particular, no vertex can ever have discrepancy at least $2.1 (\ceil{\log_2 n}+1)$, hence $d^A(v) = O(\ln n)$ for every vertex $v$ at all times.  A similar argument shows that $-d^B(v) = O(\ln n)$ for all $v$.

We now bound the length of the fractional game.  Let $\ell_o$ and $\ell_f$ denote the lengths of the ordinary and fractional games, respectively, and let $m_o$ and $m_f$ denote the numbers of brushes added by the Oracle to the ordinary and fractional games, respectively.  As argued above, when the Oracle adds brushes to a vertex $v$ in the ordinary game, it adds at most $\frac{1}{p}\max\{-d^A(v)-d^B(v),0\}$.  Furthermore, once the Oracle adds brushes to a vertex $v$, that vertex becomes clean in both games, so $d^A(v)$ and $d^B(v)$ cease to change; hence it suffices to analyze the discrepancies at the end of the game.  We now have
\begin{align*}
m_o &\le \sum_{v \in V(G)} \frac{1}{p}\max\{-d^A(v) - d^B(v),0\}\\
    &\le \frac{1}{p}\left (\sum_{v \in V(G)} \size{d^A(v)} + \sum_{v \in V(G)} \size{d^B(v)}\right )\\ 
		&\le \frac{1}{p}\left (2\sum_{v: d^A(v) > 0} d^A(v) + 2\sum_{v: d^B(v) < 0} -d^B(v)\right )\\
		&= O(n \ln n / p),
\end{align*}
where the third inequality uses the fact that $\sum_{v \in V(G)} d^A(v) = \sum_{v \in V(G)} d^B(v) = 0$.  By Lemma~\ref{lem:domination}, the Oracle's intervention changes the length of the ordinary game by only $O(n \ln n / p)$ turns.  Likewise, $m_o = O(n \ln n)$, so the Oracle's interference changes the length of the fractional game by only $O(n \ln n)$ turns.

To conclude the proof, suppose A is actually Min.  Since Min follows an optimal strategy for the ordinary game (putting aside the Oracle's interference), we have $\ell_o \le \bg(G) + O(n \ln n / p)$.  Likewise, since Max follows an optimal strategy for the fractional game, we have $\ell_f \ge \bg_p(G) - O(n \ln n)$.  The Oracle ensures that the two games finish within one round of each other, hence
$$\bg_p(G) - O(n \ln n) \le \ell_f = p\ell_o + O(1) \le p(\bg(G) + O(n \ln n / p)) + O(1) = p\bg(G) + O(n \ln n),$$
hence $\bg_p(G) \le p\bg(G) + O(n \ln n)$, which establishes the claimed upper bound on $\bg_p(G)$.

Finally, suppose A is Max.  This time, Max follows an optimal strategy for the ordinary game, so $\ell_o \ge \bg(G) - O(n \ln n / p)$.  Similarly, $\ell_f \le \bg_p(G) + O(n \ln n)$.  Combining these observations, we have
$$\bg_p(G) - O(n \ln n) \ge \ell_f = p\ell_o + O(1) \ge p(\bg(G) + O(n \ln n / p)) + O(1) = p\bg(G) + O(n \ln n),$$
hence $\bg_p(G) \ge p\bg(G) + O(n \ln n)$.  This establishes the lower bound, and completes the proof under the assumption that $1/p$ is integral.

%
Suppose now that $1/p$ is not integral.  In the argument above, after every round in the fractional game, player A simulated $1/p$ rounds in the ordinary game; when $1/p$ is not an integer, he or she cannot do this.  Instead, after round $i$ in the fractional game, A simulates $\ceil{i/p} - \ceil{(i-1)/p}$ rounds in the ordinary game.  The remainder of the argument proceeds just as before.
\end{proof}

Theorems~\ref{thm:complete} and \ref{thm:fractional} yield the following important corollary.

\begin{corollary}\label{cor:fractional_complete}
If $p \gg \ln n/n$, then $\bg_p(K_n) = (1+o(1))p\bg(K_n) = (1+o(1))pn^2/e$.
\end{corollary}

The final tool we need is the following well-known result known as the Chernoff Bound:

\begin{theorem}[\cite{JLR}]\label{thm:Chernoff}
Let $X$ be a random variable that can be expressed as a sum $X=\sum_{i=1}^n X_i$ of independent random indicator variables $X_i$, where $X_i$ is a Bernoulli random variable with $\prob{X_i = 1} = p_i$ (the $p_i$ need not be equal).  For $t \ge 0$,
\begin{eqnarray*}
\prob {X \ge \expect{X} + t} &\le& \exp \left( - \frac {t^2}{2(\expect{X}+t/3)} \right) \quad \text{ and}\\
\prob {X \le \expect{X} - t} &\le& \exp \left( - \frac {t^2}{2\expect{X}} \right).
\end{eqnarray*}
In particular, if $\eps \le 3/2$, then
\begin{eqnarray*}
\prob {|X - \expect{X}| \ge \eps \expect{X}} &\le& 2 \exp \left( - \frac {\eps^2 \expect{X}}{3} \right).
\end{eqnarray*}
\end{theorem}

We are now ready to prove the main result of this section.

\begin{proof}[Proof of Theorem~\ref{thm:random_brushing}.]
Corollary~\ref{cor:fractional_complete} establishes the second equality in the theorem statement, so it suffices to prove the first. Let $G$ be any graph on $n$ vertices. To obtain an upper bound for $\bg(G)$, we provide a strategy for Min that mimics her optimal strategy for the fractional game on $K_n$, which yields an upper bound on $\bg(G)$ in terms of $\bg_p(K_n)$. Max plays his moves in the \emph{real} game on $G$, but Min interprets them as moves in the \emph{imaginary} fractional game on $K_n$. Min plays according to an optimal strategy in the imaginary game, then makes the same move in the real game. Our hope is that the two games behave very similarly for $G \in \G(n,p)$.  

As in the proof of Theorem~\ref{thm:fractional}, we introduce an \emph{Oracle} to enforce synchronization between the games.  The Oracle can, at any time, clean a vertex in either game by adding extra brushes. Suppose some vertex fires in the real game but not in the imaginary one. When this happens, the Oracle cleans this vertex in the imaginary game; by Lemma~\ref{lem:domination}, this cannot increase the length of the imaginary game. Now suppose instead that some vertex $v$ fires in the imaginary game but not in the real one. 
In this case, the Oracle adds brushes to $v$ in the real game until it fires.  By Lemma~\ref{lem:domination}, each brush placed by the oracle can decrease the length of the real game by at most 2.

To obtain an upper bound on the length of the real game, we must bound the number of brushes added by the Oracle in the real game. Let $\pi = (v_1, v_2, \ldots, v_n)$ be the cleaning sequence produced during the game, that is, the order in which the vertices fire. Consider the state of the game when only vertices $v_1, v_2, \ldots, v_{i-1}$ are clean. In the imaginary game, $v_i$ received $(i-1)p$ brushes from earlier neighbors and needs a total of $(n-i)p$ to fire. In the real game, it received $\deg^-_{\pi}(v_i) = |N(v_i) \cap \{v_1, v_2, \ldots, v_{i-1} \}|$ brushes and needs $\deg(v_i) - \deg^-_{\pi}(v_i)$. If the Oracle adds brushes to clean $v_i$ in the real game, then it must be that, compared to the imaginary game, $v_i$ either received fewer brushes or needs more to fire (or both). The Oracle must add at most $D_{\pi}(v_i)$ brushes, where
\begin{eqnarray*}
D_\pi(v_i) &=& \max \left\{ \Big( (i-1)p - (n-i)p \Big)- \Big(\deg^-_{\pi}(v_i) - (\deg(v_i) - \deg^-_{\pi}(v_i)) \Big) , 0\right\} \\
&=& \max \left\{ \Big( 2 (i-1)p - (n-1)p \Big)- \Big(2 \deg^-_{\pi}(v_i) - \deg(v_i) \Big) , 0\right\} \\
&\le& 2 \max \left\{ (i-1)p - \deg^-_{\pi}(v_i), 0\right\} + \max \left\{ \deg(v_i) - (n-1)p, 0\right\}.
\end{eqnarray*}
Hence the length of the game is at most $\bg_p(K_n) + 2D_{\pi}(G)$, where
$$
D_{\pi}(G) = 2 \sum_{i = 1}^n \max \left\{ (i-1)p - \deg^-_{\pi}(v_i), 0\right\} + \sum_{i=1}^n \max \left\{ \deg(v_i) - (n-1)p, 0\right\}.
$$
Of course we do not know, in advance, the cleaning sequence $\pi$.  However, we still obtain the following upper bound:
$$
\bg(G) \le \bg_p(K_n) + 2\max_{\pi} D_{\pi} (G).
$$

A similar argument yields a lower bound on $\bg(G)$. We still provide a strategy for Min, but this time, she uses an optimal strategy for the real game on $G$ to guide her play in the imaginary (fractional) game on $K_n$. If a vertex fires in the imaginary game but not in the real one, then the Oracle cleans that vertex by providing extra brushes in the real game; this cannot increase the length of the real game. If a vertex $v$ fires in the real game but not in the imaginary one, then the Oracle adds cleans $v$ by adding brushes in the imaginary game; each brush added can decrease the length of the imaginary game by at most 2.  When the Oracle adds brushes in the imaginary game, it is because more brushes were received or fewer were needed (or both) in the real game than in the imaginary game. Using symmetry and the notation introduced above, we obtain $\bg_p(K_n) \le \bg(G) + 2D_{\bar{\pi}}(G)$, where $\bar{\pi}$ is the reverse of $\pi$. Consequently, we get
$$
\bg_p(K_n) \le \bg(G) + 2\max_{\pi} D_{\bar{\pi}}(G) = \bg(G) + 2\max_{\pi} D_{\pi}(G).
$$
It follows that
$$
| \bg(G) - \bg_p(K_n)| \le 2\max_{\pi} D_{\pi}(G).
$$

Now let $G \in \G(n,p)$, let $d=d(n)=p(n-1)$, and let $\omega=\omega(n)=d / \ln n$ (note that $\omega$ tends to infinity as $n \to \infty$). To complete the proof, it suffices to show that a.a.s.\ $D_{\pi}(G) = o(pn^2)$. We can easily bound the second summation in the definition of $D_{\pi}(G)$ using the Chernoff Bound.

Fix a vertex $v_i$. Clearly $\expect{\deg(v_i)} = p(n-1) = d$, so it follows immediately from the Chernoff Bound (Theorem~\ref{thm:Chernoff}) that
$$
\prob {|\deg(v_i) - d | \ge \eps d} \le 2 \exp \left( - \frac {\eps^2 d}{3} \right) = 2n^{-2},
$$
for $\eps = \sqrt{6 \ln n / d} = \sqrt{6/\omega}=o(1)$. Hence by the Union Bound, a.a.s.\ for all $i \in \{1, 2, \ldots, n\}$ we have $\deg(v_i)= (1+o(1)) d$. As a consequence, a.a.s.\
$$
\sum_{i=1}^n \max \left\{ \deg(v_i) - (n-1)p, 0\right\} = \sum_{i=1}^n \max \left\{ \deg(v_i) - d, 0\right\} \le n \cdot o(d) = o(pn^2).
$$

Estimating the first summation in the definition of $D_{\pi}(G)$ is slightly more complicated, since we must consider all possible cleaning sequences $\pi$. First, observe that the partial sum containing only the first $n/\omega^{1/5}$ terms can be bounded (deterministically, for any $\pi$) as follows:
$$
\sum_{i = 1}^{n/\omega^{1/5}} \max \left\{ (i-1)p - \deg^-_{\pi}(v_i), 0\right\} \le p \sum_{i = 1}^{n/\omega^{1/5}} (i-1) = O(p n^2 / \omega^{2/5}) = o(pn^2).
$$

Now fix some cleaning sequence $\pi$ and some $i$ exceeding $n/\omega^{1/5}$. Clearly, $\expect{\deg^-_{\pi}(v_i)} = p(i-1)$. Let $\eps = \eps(n)= \sqrt{3}/\omega^{1/5} = o(1)$. We call vertex $v_i$ \emph{bad} if $\expect{\deg^-_{\pi}(v_i)} - \deg^-_{\pi}(v_i) > \eps \expect{\deg^-_{\pi}(v_i)}$. Applying the Chernoff Bound again, the probability that $v_i$ is bad is at most
$$
\exp \left( - \frac {\eps^2 p(i-1)}{3} \right) \le \exp \left( -\omega^{-2/5} p (n \omega^{-1/5}) \right) \le \exp \left( -\omega^{2/5} \ln n \right) =: q.
$$
The important observation is that the events $\{ v_i \text{ is bad}\}$ are mutually independent, so the probability that there are at least $n/\omega^{1/5}$ bad vertices is at most
$$
2^n q^{n/\omega^{1/5}} = \exp \left( O(n) -\omega^{1/5} n \ln n \right) = \exp \left(- (1+o(1)) \omega^{1/5} n \ln n \right) = o(1/n^n) = o(1/n!).
$$
Hence, by the Union Bound, a.a.s.\ for all possible cleaning sequences $\pi$ there are at most $n/\omega^{1/5}$ bad vertices. It follows that a.a.s.\ the first summation in the definition of $D_{\pi}(G)$ can be bounded as follows:
\begin{eqnarray*}
\sum_{i = 1}^n \max \left\{ (i-1)p - \deg^-_{\pi}(v_i), 0\right\} &=& o(pn^2) + \sum_{i > n/\omega^{1/5}} \max \left\{ (i-1)p - \deg^-_{\pi}(v_i), 0\right\} \\
&\le& o(pn^2) + n (\eps d) + (n/\omega^{1/5}) (np) = o(pn^2).
\end{eqnarray*}
This completes the proof.
\end{proof}

\end{document}